\documentclass[reqno,11pt]{amsart} 

\usepackage{amsmath,amssymb,amsthm,mathtools}
\usepackage{mathrsfs}       
\usepackage{hyperref}
\usepackage[nameinlink]{cleveref}
\usepackage{xcolor}         
\usepackage{tikz}           
\usepackage{verbatim}       
\usepackage{setspace}       
\usepackage[normalem]{ulem}
\title[Borderline regularity in singular FBP]{Borderline regularity in singular\\ free boundary problems}

\author[D.J. Araújo]{Dami\~ao J. Ara\'ujo}
\address{Department of Mathematics, Universidade Federal da Para\'iba, 58059-900, Jo\~ao Pessoa-PB, Brazil}
\email{araujo@mat.ufpb.br}

\author[A. Sobral]{Aelson Sobral}
\address{Applied Mathematics and Computational Sciences (AMCS), Computer, Electrical and Mathematical Sciences and Engineering Division (CEMSE), King Abdullah University of Science and Technology (KAUST), Thuwal, 23955-6900, Kingdom of Saudi Arabia}
\email{aelson.sobral@kaust.edu.sa}

\author[E. V. Teixeira]{Eduardo V. Teixeira}
\address{Department of Mathematics, Oklahoma State University, 401 Mathematical Sciences, Stillwater, OK 74078, USA}
\email{eduardo.teixeira@okstate.edu}

\author[J.M. Urbano]{Jos\'{e} Miguel Urbano}
\address{Applied Mathematics and Computational Sciences (AMCS), Computer, Electrical and Mathematical Sciences and Engineering Division (CEMSE), King Abdullah University of Science and Technology (KAUST), Thuwal, 23955-6900, Kingdom of Saudi Arabia and CMUC, Department of Mathematics, University of Coimbra, 3000-143 Coimbra, Portugal}
\email{miguel.urbano@kaust.edu.sa}

\newtheorem{theorem}{Theorem}[section]
\newtheorem{lemma}{Lemma}[section]

\newtheorem{proposition}{Proposition}[section]

\newtheorem{remark}{Remark}[section]

\newcommand{\intav}[1]{\mathchoice 
  {\mathop{\vrule width 6pt height 3 pt depth -2.5pt \kern -8pt \intop}\nolimits_{\kern -6pt#1}} 
  {\mathop{\vrule width 5pt height 3 pt depth -2.6pt \kern -6pt \intop}\nolimits_{#1}}
  {\mathop{\vrule width 5pt height 3 pt depth -2.6pt \kern -6pt \intop}\nolimits_{#1}}
  {\mathop{\vrule width 5pt height 3 pt depth -2.6pt \kern -6pt \intop}\nolimits_{#1}}}

\numberwithin{equation}{section}
\begin{document}

\subjclass[2020]{Primary 35B65. Secondary 35R35, 35A21}




\keywords{Borderline regularity, Free boundary problems, Non-differentiable functionals}

\begin{abstract}
In this paper, we investigate the borderline regularity of local minimizers of energy functionals under minimal assumptions on the potential term $\sigma$. When $\sigma$ is merely bounded and measurable, we show that sign-changing minimizers are Log-Lipschitz continuous, which represents the optimal regularity in this general setting. In the one-phase case, however, we establish gradient bounds for minimizers along their free boundaries, revealing a structural gain in regularity. Most notably, we prove that if $\sigma$ is continuous, then minimizers are of class $C^1$ along the free boundary, thereby identifying a sharp threshold for differentiability in terms of the regularity of the potential.
\end{abstract}

\date{\today}

\maketitle

\section{Introduction} \label{sct intro}

In this paper, we address borderline regularity issues for local minimizers of energy functionals of the form
\begin{equation}\label{main functional}
\mathcal{J}(u,\Omega) \coloneqq \int_{\Omega} \frac{1}{2}|Du|^2 + \sigma(u)\,dx,
\end{equation}
where $\Omega \subset \mathbb{R}^n$ is a bounded domain, and $\sigma \colon \mathbb{R} \to \mathbb{R}$ is a  bounded measurable function. Functionals of this type arise in a variety of contexts, including flame propagation, image segmentation, and biological modeling. A critical aspect of such models relies on the non-differentiability of $\sigma$ at a collection of points $\mathscr{S} \subset \mathbb{R}$, leading to the formation of free boundaries
$$\Gamma \coloneqq \partial \left [u^{-1}(\mathbb{R} \setminus \mathscr{S}) \right ] \cap \Omega.$$ 

When $\sigma$ is differentiable on a region $\mathcal{D} \subset \mathbb{R}$, it is well known that local minimizers satisfy the equation  
\begin{equation}\label{Euler-Lagangre EQ}
    \Delta u = \sigma'(u) \qquad \text{in} \quad \left\{ x \in \Omega \colon u(x) \in \mathcal{D} \right\},
\end{equation}  
in an appropriate weak sense. Some of the most prominent examples include the obstacle problem \cite{CAF1}, $\sigma(t) = t_{+}$, the Alt--Phillips (or quenching) problem \cite{AP}, $\sigma(t) = t_+^\gamma$, for $0<\gamma < 1$, the Alt--Caffarelli (or cavity) problem \cite{AC}, $\sigma(t) = \chi_{\{t > 0\}}$, the theory of flame propagation \cite{BCN, LW1}, $\sigma(t) = \sigma_\epsilon(t) = \epsilon^{-1} \beta(\epsilon^{-1} t)$, among several others.  

For non-differentiable $\sigma$, the minimization property of the functional $\mathcal{J}$ does not yield an associated Euler--Lagrange equation. Nevertheless, a significant part of the regularity theory can still be developed based solely on the variational structure, which often carries more information than the Euler--Lagrange equation \eqref{Euler-Lagangre EQ} alone. The investigation of such functionals has a long history, and many fundamental results extend to the setting of almost-minimizers or quasi-minimizers; see, for instance, \cite{GG, DT, DSS, DFFV, PT24}. 

In \cite{DeFM}, a comprehensive regularity theory was developed for local minimizers of variational energies, including sharp gradient bounds, under the assumption that the function $\sigma$ is H\"older continuous. When $\sigma$ is merely bounded, one can still derive interior regularity estimates of class $C^{0,\alpha}$, and even $C^{\text{Log-Lip}}$, which are essentially optimal in such generality, see \Cref{Reg 2-phase}. The main goal of this work, however, is to advance a borderline regularity theory for one-phase local minimizers of \eqref{main functional}, under minimal structural assumptions on the potential $\sigma$.

Our first main result shows that if $\sigma$ is merely bounded and measurable, then nonnegative local minimizers are Lipschitz continuous at the free boundary. We then go beyond this threshold: assuming $\sigma$ is continuous, we prove that minimizers are $C^1$ at the free boundary, with a modulus of continuity that depends explicitly -- through an intricate nonlinear process -- on that of $\sigma$. Crucially, no structural conditions are imposed on the continuity of $\sigma$: it need not be H\"older or Dini continuous. We construct the modulus directly, uncovering a precise mechanism through which the continuity of $\sigma$ is transmitted to the free boundary. 

These results identify the exact threshold where minimizers become more regular, showing how the structure of $\sigma$ influences the smoothness of solutions along the free boundary. A central insight of this work is identifying how this regularity arises. Near free boundary points, minimizers of \eqref{main functional} must vanish. Leveraging the scale invariance of the functional, rescaled solutions flatten in a controlled way. This flattening becomes more noticeable at smaller scales, highlighting a subtle but strong regularizing effect built into the variational structure of the problem, which ultimately leads to gradient bounds of minimizers when $\sigma$ is merely bounded.

To advance beyond Lipschitz regularity and reach $C^1-$smoothness, we introduce a nonlinear renormalization scheme. This iterative approach rescales the problem at dyadic levels, each time adjusting the normalization to preserve a smallness condition essential for convergence. A key insight is that the continuity of $\sigma$ -- regardless of its modulus -- ensures coherence across scales, allowing the scheme to bridge the gap between geometric flatness and true differentiability. This strategy not only establishes $C^1-$regularity but also reveals the critical role played by the continuity of $\sigma$ in propagating smoothness to the free boundary. This regularity phenomenon is, at least heuristically, reminiscent of the effect observed in \cite{SH03}, which was achieved through a clever application of the powerful Caffarelli-Jerison-Kenig almost monotonicity formula \cite{CJK}.

The techniques developed in this paper are robust and could be extended to a broader class of elliptic operators beyond the Laplacian. However, for the sake of clarity and to streamline the exposition, we have chosen to present our results in the context of the classical Laplace operator. This choice allows us to focus on the core ideas without the additional technical overhead that more general operators would entail.

The remainder of the paper is organized as follows. In \Cref{prelim-sect}, we introduce the mathematical framework, discuss the relevant scaling properties, and review key results from the literature that will be used throughout the paper. \Cref{Reg 2-phase} is devoted to establishing regularity estimates for the two-phase problem. Sections~\ref{LIP-sct} and~\ref{C1 - sct} focus on borderline regularity results for nonnegative minimizers: in the former, we prove Lipschitz continuity at the free boundary, while in the latter, we establish $C^1-$regularity at the free boundary under continuity assumptions on the potential $\sigma$.

\section{Preliminary results}\label{prelim-sect}

In this section, we collect some preliminary results needed for the remainder of the paper. 

\subsection{Mathematical setup} 

Let $\Omega \subset \mathbb{R}^n$ be a bounded, smooth domain and $\sigma \colon \Omega \times \mathbb{R} \to \mathbb{R}$ be a bounded measurable function in both its variables. Although our main interest lies in $\sigma-$functions without explicit dependence on $x$, as in functional \eqref{main functional}, the proof of the Log-Lipschitz estimates given in the next section requires considering such $x-$dependence as part of an auxiliary problem. For this reason, we consider this slightly more general setting.

For a nonnegative boundary datum $0 \leq \varphi \in H^1(\Omega) \cap L^\infty(\Omega)$, we consider the problem of minimizing the functional 
\begin{equation}\label{fifi}
	\mathcal{J}(v, \Omega) \coloneqq \int_{\Omega} \frac{1}{2} \left| Dv \right|^2 + \sigma(x,u)\,dx
\end{equation}
among competing functions 
$$
    v \in {\mathcal A} \coloneqq \left\{ v \in H^1(\Omega) \ \colon \ v-\varphi \in H_0^1(\Omega) \right\}. 
$$
We say $u \in \mathcal{A}$ is a minimizer of \eqref{fifi} if 
$$
    \mathcal{J}(u, \Omega) \leq \mathcal{J}(v, \Omega), \quad \forall v \in \mathcal{A}. 
$$
Note that minimizers as above are, in particular, local minimizers in the sense that, for any open subset $\Omega^\prime \subset \Omega$,
$$
    \mathcal{J}(u, \Omega^\prime) \leq \mathcal{J}(v, \Omega^\prime), \quad \forall v \in H^1(\Omega^\prime)\ \colon \ v-u \in H_0^1(\Omega^\prime).
$$
The theory developed in this article applies to both notions, but the focus of our results will, from now on, be on local minimizers.

Throughout the paper, we say that a constant $C$ is universal if it depends only on the data, namely, the dimension $n$, the $L^\infty-$norm of $\sigma$, and the boundary data $\varphi$. Any additional dependencies will be mentioned explicitly in the results' statements.

\subsection{Scaling}\label{scaling of functional}

Most of the arguments recurrently used in this paper rely on a scaling feature of the functional \eqref{fifi} that we detail in the sequel for future reference. Let $x_0 \in \Omega$ and consider two parameters $a,b \in (0,1]$. If $u\in H^1(\Omega)$ is a minimizer of $\mathcal{J}(v,B_a(x_0))$, then
\begin{equation}\label{august}
    w(x) \coloneqq \frac{u(x_0 + ax)}{b}, \quad x\in B_1
\end{equation}
is a minimizer of the functional
$$
    \tilde{\mathcal{J}}(v,B_1) \coloneqq \int_{B_1} \frac{1}{2} \left| Dv \right|^2 + \tilde{\sigma}(x,v) dx,
$$
with
$$
    \tilde{\sigma}(x,t) \coloneqq \left(\frac{a}{b} \right)^2  \sigma(x_0 + ax,bt).
$$   
Indeed, by changing variables,
$$
    \int_{B_a(x_0)} \frac{1}{2} \left| Du (x)\right|^2 + \sigma(x,u(x))\,dx  
$$
\begin{eqnarray*}
 &= &  a^{n} \int_{B_1} \frac{1}{2} \left| Du(x_0 + ay) \right|^2 + \sigma(x_0 + ay,u(x_0 + ay))\,dy\\
 &   = &  a^n \int_{B_1} \frac{1}{2} \left| \left(\frac{b}{a} \right) Dw (x) \right|^2 + \sigma(x_0 + ax, b w(x))\,dx \\
 &   = & a^{n-2}b^2 \int_{B_1} \frac{1}{2} \left|Dw (x) \right|^2 + b^{-2}a^2  \sigma(x_0 + ax, b w(x)) dx \\
 &   = & a^{n-2}b^2 \int_{B_1} \frac{1}{2} \left| Dw (x) \right|^2 + \tilde{\sigma}(x,w(x))dx.
\end{eqnarray*}
In particular, choosing $a=r$, $b=r^\beta$, for any $\beta \in [0,1]$, $0<r\leq 1$, we have 
$$
    \left|\tilde{\sigma}(x,t)\right| \leq \left|\sigma \left( x_0 + rx, r^\beta t \right) \right|.
$$
Recall that no sign restrictions are imposed on $\sigma$. 

We also need the following simple result regarding translation by affine functions.

\begin{lemma}\label{equation for rescaled function}
Let $x_0 \in B_1$ and $a,b$ be two positive constants, and define
$$
    v(x) \coloneqq \frac{(u-\ell)(x_0 + ax)}{b},
$$
where $\ell$ is an affine function. Then, $u$ is a local minimizer of \eqref{fifi} in $B_a(x_0)$, if and only if, $v$ is a local minimizer of
$$
    \int_{B_1} \frac{1}{2}|Dw|^2 + \sigma_{\ast}(x,w)\,dx,
$$
where $\sigma_{\ast}(x,t) \coloneqq b^{-2}a^{2}\sigma\left(x_0 + ax, bt + \ell(ax)\right)$.
\end{lemma}

\begin{proof}
By the scaling feature of the functional, it is enough to consider $a=b=1$ and $x_0 = 0$. We show that if $w$ is such that $(v-w) \in H^1_0(B_1)$, then 
\begin{equation} \label{HongKong}
    \int_{B_1} \frac{1}{2}|Dv|^2 + \sigma_{\ast}(x,v)\,dx \leq \int_{B_1} \frac{1}{2}|Dw|^2 + \sigma_{\ast}(x,w)\,dx.
\end{equation}
Observe that the function $\overline{w} = w + \ell$ is a valid competitor in the minimization problem for $u$, and so
$$
    \int_{B_1}\frac{1}{2}|Du|^2 + \sigma(x,u)\,dx \leq \int_{B_1}\frac{1}{2}|D(w + \ell)|^2 + \sigma(x,w + \ell)\,dx.
$$
In terms of $v$, this reads
$$
    \int_{B_1}\frac{1}{2}|D(v + \ell)|^2 + \sigma_{\ast}(x,v)\,dx \leq \int_{B_1}\frac{1}{2}|D(w + \ell)|^2 + \sigma_{\ast}(x,w)\,dx,
$$
which implies \eqref{HongKong}. Indeed, the extra term
$$
    \int_{B_1}D(v-w) \cdot D\ell\, dx
$$
vanishes, which follows from integrating by parts and using that $v - w \in H^1_0(B_1)$, and that $\Delta \ell = 0$ because $\ell$ is affine. 
\end{proof}

\subsection{Useful estimates}

We gather here some useful estimates, which can be found in \cite[Lemma 2.4 and Lemma 4.1, respectively]{LQT}. We adjust the statements of the lemmata to fit the setup treated here. 

Given a ball $B_R(x_0) \Subset \Omega$, we denote the harmonic replacement of $u$ in $B_R(x_0)$ by $h$, \textit{i.e.}, $h$ is the solution of the boundary value problem
$$
	\Delta h=0 \; \text{ in }\; B_R(x_0)  \quad \text{ and } \quad h-u \in H^1_0(B_R(x_0)).
$$ 
By the maximum principle, we have
\begin{equation}\label{harmonic}
\|h\|_{L^\infty(B_R(x_0))} \leq \|u\|_{L^{\infty}(B_R(x_0))}.
\end{equation}

\begin{lemma}\label{ineq for harmonic replac}
Let $\psi \in H^1(B_R)$ and $h$ be the harmonic replacement of $\psi$ in $B_R$. There exists $c$, depending only on $n$, such that 
\begin{equation}
c\int\limits_{B_R}|D\psi-Dh|^2 \, dx \leq \int\limits_{B_R}|D\psi|^2-|Dh|^2 \, dx. 
\end{equation}
\end{lemma}

\begin{lemma}\label{decay in terms of harmonic replac}
Let $\psi \in H^1(B_R)$ and $h$ be the harmonic replacement of $\psi$ in $B_R$. Given $\beta \in (0,1)$, there exists $C$, depending only on $n$ and $\beta$, such that 
\begin{eqnarray*}
\int\limits_{B_r}|D\psi-(D\psi)_r|^2 \, dx & \leq & C\left( \frac{r}{R}\right)^{n+2\beta} \int\limits_{B_R}|D\psi-(D\psi)_R|^2 \, dx \\
& & + C\int\limits_{B_R}|D\psi-Dh|^2 \, dx,
\end{eqnarray*}
for each $0<r \leq R$. 
\end{lemma}

We conclude this subsection with the $L^2$ to $L^\infty$ estimates. We refer to \cite[Theorem 7.2]{EG} for a proof.

\begin{lemma}\label{2 to infty bddness}
Assume $u$ is a local minimizer to \eqref{main functional} in $B_1$ with $\sigma$ bounded. Then, there exists a universal constant $C>0$ such that
$$
    \|u\|_{L^\infty(B_{1/2})} \leq C\|u\|_{L^2(B_1)}.
$$
\end{lemma}

\section{Regularity for two-phase minimizers} \label{Reg 2-phase}

In this intermediary section, we first establish, with $\sigma$ merely bounded, the interior $C^{0,\alpha}$ and $C^{0,\text{Log-Lip}}-$regularity of two-phase local minimizers of \eqref{main functional}. While similar results may be found in the literature, we include them here for two main reasons: first, to emphasize that the regularity estimates rely only on the boundedness of $\sigma$; and second, to provide a simplified and self-contained proof of the Log-Lipschitz estimate. For a H\"older continuous $\sigma$, we revisit the known $C^{1,\alpha}-$regularity result for local minimizers at the end of the section.

\begin{proposition}\label{interior holder regularity}
Let $u$ be a local minimizer of \eqref{main functional} in $B_1$. Given $\alpha \in (0,1)$, there exists a constant $C>0$, depending only on $\alpha$ and the data, such that
$$
    \|u\|_{C^{0,\alpha}(B_{1/2})} \leq C\|u\|_{L^2(B_1)}.
$$
\end{proposition}

\begin{proof}
First, after rescaling, we may assume that $\|u\|_{L^2(B_1)} \leq 1$ and 
$$\|\sigma\|_{L^\infty(\mathbb{R})} \leq \delta,$$ 
for some small $\delta>0$. Let $h$ be the harmonic replacement of $u$ in $B_1$. Using it as a competitor for $u$ yields
$$
  \int_{B_1}|Du|^2 - |Dh|^2 \leq 2\int_{B_1}\sigma(h) - \sigma(u).  
$$
By \Cref{ineq for harmonic replac}, and since both $u$ and $h$ are bounded (by the boundary data, via the maximum principle), we have
$$
    \int_{B_1}|D(u-h)|^2 \leq C\delta.
$$
Since $u-h \in H^1_0(B_1)$, we can use Poincaré inequality in balls to obtain
$$
    \int_{B_1}|u-h|^2 \leq C\delta.
$$
Now, for $r<1$ small, we observe that 
\begin{eqnarray*}
    \int_{B_r}|u(x) - h(0)|^2 & \leq & 2\int_{B_r}|u(x) - h(x)|^2 + 2\int_{B_r}|h(x) - h(0)|^2\\
    & \leq & 2C\delta + 2C^{\prime}\|h\|^2_{L^\infty(B_1)}r^{n+2}\\
    & \leq & r^{2\alpha + n},
\end{eqnarray*}
for $r$ and $\delta$ small enough depending also on $\alpha$.

Now we iterate this reasoning. We will build a sequence $(a_k)_{k \in \mathbb{N}}$ such that, for every $k \in \mathbb{N}$, 
\begin{equation}\label{integral decay}
    \intav{B_{r^k}}|u - a_k|^2\,dx \leq r^{2\alpha k} \qquad \text{and} \qquad |a_{k} - a_{k-1}| \leq Cr^{\alpha (k-1)}.
\end{equation}
We proceed by induction. The case $k=1$ follows by considering $a_1 = h(0)$ and $a_0=0$. Now, assuming we have found $a_k$, we define the rescaled function
$$
    w_k(x) \coloneqq \frac{u(r^k x) - a_k}{r^{k\alpha}}.
$$
We point out that $w_k$ is a local minimizer of the functional
$$
    \int_{B_1}\frac{1}{2}|Dw|^2 + \sigma_k(w)\,dx,
$$
where $\sigma_k(t) = r^{2k(1-\alpha)}\sigma(r^{k\alpha}t + a_k)$. By the induction assumption, it follows that
$$
    \intav{B_1}|w_k|^2\,dx \leq 1,
$$
and so, by \Cref{2 to infty bddness}, it follows that $w_k$ is uniformly bounded. Repeating the argument at the beginning of the induction, we find $\overline{a} = \overline{h}(0)$, where $\overline{h}$ is the corresponding harmonic replacement for $w_k$. It then follows that
$$
    \int_{B_r}|w_k - \overline{a}|^2 \leq r^{k(2\alpha+n)},
$$
which in turn means
$$
    \int_{B_{r^{k+1}}}|u - a_{k+1}|^2 \leq r^{(k+1)(2\alpha+n)},
$$
where $a_{k+1} = a_k + r^{k\alpha}\overline{a}$. By \eqref{integral decay}, the sequence $(a_k)_{k \in \mathbb{N}}$ is a Cauchy sequence, and thus converges. From here, the argument is standard.
\end{proof}

A refinement of the previous argument suggests the possibility of obtaining a Log-Lipschitz estimate. This is typically achieved by showing that $Du$ belongs to BMO, which is the energetic analogue of Log-Lipschitz continuity. Here, we present a simpler proof of the Log-Lipschitz estimate. To this end, we consider functionals of the form \eqref{fifi}, allowing for $x-$dependence in the coefficient function $\sigma$. The key observation is that, despite this additional dependence, $\sigma$ remains bounded.

We begin with the following approximation lemma.

\begin{lemma}\label{flatland log lip}
Let $u$ be a local minimizer of \eqref{fifi} in $B_1$. There exist universal parameters $r, \delta \in (0,1)$ and an affine function $\ell(x) \coloneqq a + b \cdot x$, with universally bounded coefficients such that if 
$$
    \|u\|_{L^\infty(B_1)} \leq 1 \quad \text{and} \quad \|\sigma\|_{L^\infty(\mathbb{R})} \leq \delta,
$$
then
$$
    \int_{B_r}|u - \ell|^2 \leq r^{n+2}.
$$
\end{lemma}

\begin{proof}
Let $r>0$ be a small parameter to be chosen later, and let $h$ be the harmonic replacement of $u$ in $B_1$. Using $h$ as a competitor for $u$ yields
$$
  \int_{B_1}|Du|^2 - |Dh|^2 \leq 2\int_{B_1}\sigma(x,h) - \sigma(x,u).  
$$
By \Cref{ineq for harmonic replac}, and since both $u$ and $h$ are bounded, we have
$$
    \int_{B_1}|D(u-h)|^2 \leq C\|\sigma\|_{L^\infty(\mathbb{R})},
$$
which, combined with the Poincar\'e inequality in balls, gives
$$
    \int_{B_1}|u-h|^2 \leq C\|\sigma\|_{L^\infty(\mathbb{R})}.
$$
Since $h$ is harmonic in $B_1$, we obtain
$$
    |h(x) - (h(0) + Dh(0)\cdot x)| \leq C_1 \|h\|_{L^\infty(B_1)} |x|^2 \quad \text{in} \ B_{1/2},
$$
and so, if we consider $\ell(x) \coloneqq h(0) + Dh(0) \cdot x$, we obtain
$$
    \int_{B_r}|u-\ell|^2 \leq 2C\delta + 2C_2 \|h\|^2_{L^\infty(B_1)} r^{n+4}.
$$
By assumption, $\|u\|_{L^\infty(B_1)} \leq 1$, which also implies $\|h\|_{L^\infty(B_1)} \leq 1$. We can then pick $r$ universally small so that
$$
    2C_2 \|h\|^2_{L^\infty(B_1)} r^{n+4} \leq \frac{r^{n+2}}{2},
$$
and then $\delta$ small so that
$$
    2C \delta \leq \frac{r^{n+2}}{2},
$$
from which the result follows.
\end{proof}

A careful iteration of the approximation lemma gives the Log-Lipschitz estimate.

\begin{proposition}\label{log lipschitz regularity}
Let $u$ be a local minimizer of \eqref{main functional} in $B_1$. Then, there exists a universal constant $C>0$ such that
$$
    |u(x) - u(y)| \leq C\|u\|_{L^2(B_1)}|x-y|\ln\left(|x-y|^{-1}\right),  \quad x,y \in B_{1/2}.
$$
\end{proposition}

\begin{proof}
By scaling, we may assume $\sigma$ is in a smallness regime and that $u$ is normalized in $L^2$. The proof consists in finding a sequence of affine functions $\ell_k(x) \coloneqq a_k + b_k \cdot x$ such that
$$
    \int_{B_{r^k}}|u - \ell_k|^2 \leq r^{k(n+2)},
$$
and
\begin{equation}\label{control of coefs}
    |a_{k+1} - a_k| \leq Cr^{k} \quad \text{and} \quad |b_{k+1} - b_k| \leq C,
\end{equation}
for some $r>0$, $C>0$ universal and for every $k \in \mathbb{N}$. This is proved by induction. The first case is trivial by the normalization assumption. Assuming this holds up to $k$, we define
$$
    v_k(x) \coloneqq \frac{(u-\ell_k)(r^k x)}{r^k},
$$
which satisfies, by the induction assumption,
$$
    \intav{B_1}|v_k|^2 \leq 1.
$$
\Cref{2 to infty bddness} then implies that $v_k$ is uniformly bounded. We observe that $v_k$ minimizes a functional of a similar type by \Cref{equation for rescaled function}. We can then apply \Cref{flatland log lip} to obtain
$$
    \int_{B_r}|v_k - \ell|^2 \leq r^{n+2},
$$
which is equivalent to
$$
    \int_{B_{r^{k+1}}}\left|u(y) - \left(\ell_k(y) + r^k\ell(r^{-k}y)\right)\right|^2 \leq r^{(k+1)(2+n)}.
$$
Defining $\ell_{k+1} \coloneqq \ell_k + r^k\ell(r^{-k}-)$, it is standard to verify condition \eqref{control of coefs}. We observe now that for every $k \in \mathbb{N}$, the function $v_k$, previously defined, is a minimizer to a functional with a similar structure and normalized $L^2-$norm. We apply \Cref{2 to infty bddness} to obtain
$$
    \|v_k\|_{L^\infty(B_1)} \leq C,
$$
for some constant $C$,
depending on universal parameters. But this is equivalent to
$$
    \|u - \ell_k\|_{L^\infty(B_{r^k})} \leq Cr^k, \quad \text{for} \quad k \in \mathbb{N}.
$$
This condition, together with \eqref{control of coefs}, implies the Log-Lipschitz estimate (see \cite{ET14}).
\end{proof}

\begin{remark}
The fact that the regularity estimates depend solely on the boundedness of $\sigma$ is relevant for the compactness of minimizers across free boundary models. For instance, let $0<\gamma<1$ and $u_\gamma$ be a local minimizer of the two-phase Alt--Phillips functional
$$
    \int_{B_1} \frac{1}{2}|Dv|^2 + \Lambda_1 v_{-}^\gamma + \Lambda_2 v_{+}^\gamma \,dx,
$$
for positive constants $\Lambda_1$ and $\Lambda_2$. A consequence of Proposition \ref{log lipschitz regularity} is that the family $(u_\gamma)_{\gamma > 0}$ is pre-compact in the uniform convergence topology. This gives enough compactness to obtain, by passing to the limit as $\gamma \to 0$, a Log-Lipchitz solution of the celebrated two-phase cavity problem investigated in \cite{CJK}, via the Alt--Phillips problem.
\end{remark}

When $\sigma$ is H\"older continuous, the regularity theory reaches the $C^{1,\alpha}$ threshold. The heart of the matter is that the functional \eqref{main functional} behaves as that of the Alt--Phillips problem. We refer to \cite{DeFM} for a complete description and extension of such regularity results. The proof we include below is a simplified version of \cite[Theorem 2.1]{ASTU}, and we include it here just for the sake of completeness as a courtesy to the reader.

\begin{proposition}\label{holder sigma}
Assume $u$ is a local minimizer to \eqref{main functional} in $B_1$. If $\sigma$ is H\"older, then there is an exponent $\alpha$, depending on data and the H\"older continuity of $\sigma$, such that
$$
    \|u\|_{C^{1,\alpha}(B_{1/2})} \leq C\|u\|_{L^\infty(B_1)},
$$
for some universal constant $C>0$.
\end{proposition}

\begin{proof}
We prove the result for the case of balls $B_R(x_0) \Subset B_{1/2}$.
Without loss of generality, assume $x_0=0$ and denote $B_R\coloneqq B_R(0)$. 
Since $u$ is a local minimizer, by testing \eqref{main functional} against its harmonic replacement, we obtain the inequality
\begin{equation} \label{fermi}
\displaystyle\int\limits_{B_R}|Du|^2-|Dh|^2 \, dx  \leq  2 \displaystyle\int\limits_{B_R}  \sigma(h(x))-\sigma(u(x)) \, dx.
\end{equation}
By the H\"older assumption on $\sigma$ we get
$$
    \displaystyle\int\limits_{B_R}  \sigma(h(x))-\sigma(u(x)) \, dx \leq C \int\limits_{B_R} |u(x)-h(x)|^{\gamma}\, dx,
$$
for some $\gamma > 0$. In addition, by combining H\"older  and Sobolev inequalities, we obtain
\begin{eqnarray}
    \int\limits_{B_R} |u-h|^{\gamma}\, dx & \leq & C |B_R|^{1-\frac{\gamma}{2^{\ast}}}\left( \,\int\limits_{B_R} |u-h|^{2^{\ast}}\, dx \, \right)^{\frac{\gamma}{2^{\ast}}} \nonumber \\
& \leq & C |B_R|^{1-\frac{\gamma}{2^{\ast}}}\left( \displaystyle \int\limits_{B_R} |Du-Dh|^{2}\, dx \right)^{\frac{\gamma}{2}}, \label{trieste}
\end{eqnarray}
for $2^{\ast}=\dfrac{2n}{n-2}$. 

Therefore, using \Cref{ineq for harmonic replac}, together with \eqref{fermi} and \eqref{decay in terms of harmonic replac}, we get
\begin{equation}\label{Cariri7}
\int\limits_{B_R}|Du-Dh|^2 \, dx \leq C |B_R|^{\frac{2(2^{\ast}-\gamma)}{2^{\ast}(2-\gamma)}}=CR^{n+2\frac{\gamma}{2-\gamma}}.
\end{equation}
Finally, by taking 
$$
\epsilon=\frac{\gamma}{2-\gamma} \in (0,1)
$$
in \Cref{decay in terms of harmonic replac}, we conclude 
$$\int\limits_{B_r}|Du-(Du)_r|^2 \, dx $$
$$\leq C\left( \frac{r}{R}\right)^{n+2\frac{\gamma}{2-\gamma}} \int\limits_{B_R}|Du-(Du)_R|^2 \, dx + CR^{n+2\frac{\gamma}{2-\gamma}},
$$
for each $0<r \leq R$. Campanato's embedding theorem completes the proof.
\end{proof}

\section{Lipschitz regularity}\label{LIP-sct}

In this section, under a mere boundedness assumption on $\sigma$, we prove Lipschitz regularity along free boundary points of one-phase local minimizers of \eqref{main functional}. The main idea is that local minimizers are universally flat near free boundary points.

\begin{lemma}\label{flat lemma}
Given $\epsilon > 0$, there is $\delta>0$, depending only on $\epsilon$, such that for any $u$ nonnegative local minimizer of \eqref{main functional} in $B_1$, with 
$$
   u(0) = 0, \quad 0 \leq u \leq 1, \text{ and }\, \|\sigma\|_{L^\infty(\mathbb{R})} \leq \delta,
$$
there holds
$$
    \sup_{B_{1/2}} u \leq \epsilon.
$$
\end{lemma}

\begin{proof}
The proof is by compactness and is fairly classical, so we only outline it here. Assume the lemma is false. Then, there would be $\epsilon_0 > 0$ and a sequence $(u_k, \sigma_k)_{k \in \mathbb{N}}$, where $u_k$ is a minimizer to the functional
$$
    \mathcal{J}_k(v,B_1) = \int_{B_1} \frac{1}{2}|Dv|^2 + \sigma_k(v)\,dx,
$$
with $u_k(0) = 0$, $0 \leq u_k \leq 1$ in $B_1$, and $\|\sigma_k\|_{L^\infty(\mathbb{R})} \leq k^{-1}$, but
\begin{equation} \label{zhou}
    \sup_{B_{1/2}}u_k > \epsilon_0.
\end{equation}
It then follows that $\left( u_k \right)_k$ is equicontinuous by \Cref{log lipschitz regularity}, and converges to some function $u_\infty$ that minimizes
$$
    \mathcal{J}_\infty(v,B_1) \coloneqq \int_{B_1} \frac{1}{2}|Dv|^2\, dx,
$$
with $u_\infty \geq 0$ and $u_\infty(0) = 0$. By the strong maximum principle, this can only happen if $u_\infty \equiv 0$, which contradicts \eqref{zhou}.
\end{proof}

A careful iteration of the flatness lemma gives the Lipschitz continuity at the free boundary.

\begin{theorem}\label{Lip at FB}
Let $u$ be a nonnegative local minimizer of \eqref{main functional} in $B_1$ and $x_0 \in \partial \{u>0\} \cap B_{1/2}$. Then,
$$
    |u(x)| \leq C \left( \|u\|_{L^\infty(B_{3/4})} + \sqrt{\|\sigma\|_{L^\infty(\mathbb{R})}}\, \right) |x - x_0|, \quad x \in B_{1/4}(x_0), 
$$
for some universal constant $C>0$.
\end{theorem}

\begin{proof}
Fix $x_0 \in \partial \{u>0\} \cap B_{1/2}$. Let $\epsilon = 2^{-1}$ and take the corresponding $\delta >0$ given by \Cref{flat lemma}. Up to replacing $u$ by
$$
    v(x) \coloneqq \frac{u(x_0+2^{-1}x)}{\|u\|_{L^\infty(B_{3/4})} + \sqrt{\delta^{-1}\|\sigma\|_{L^\infty(\mathbb{R})}}}, \quad x \in B_1,
$$
we may assume $x_0 = 0$, and
$$
   u(0) = 0, \quad 0 \leq u \leq 1, \text{ and }\, \|\sigma\|_{L^\infty(\mathbb{R})} \leq \delta.
$$
We claim
$$
    \sup_{B_{2^{-k}}} u \leq 2^{-k}, \quad \text{for all} \quad k \in \mathbb{N}_0. 
$$
It is clear that once the claim holds, $u$ is Lipschitz regular at $x=0$.

To prove the claim, we proceed by induction. The case $k = 0$ is trivial since $u$ is normalized. We then assume it holds for $k$ and prove the same is true for $k+1$. Define
$$
    v_k(x) \coloneqq 2^ku(2^{-k}x), \quad x \in B_1.
$$
By the scaling feature of the functional (see \Cref{scaling of functional}), $v_k$ minimizes a functional of the same type, with its corresponding $\sigma$ factor still in a smallness regime. Moreover, $v_k$ is normalized by the induction assumption and $v_k(0) = 0$. We can then use \Cref{flat lemma} to obtain
$$
    \sup_{B_{1/2}}v_k \leq 2^{-1},
$$
from which the induction step follows. The claim is then proven.
\end{proof}

\begin{remark}
We comment that once gradient bounds are established along the free boundary, one can often obtain local Lipschitz regularity, provided the governing PDE, $\Delta u = \sigma'(u)$ in $\{u > 0\}$, allows for such smoothness. This is the case of the classical Alt--Caffarelli functional; see the details in \cite{PT16}.
\end{remark}

\begin{remark}
We emphasize that, given the generality of our results, the free boundary 
$$
    \Gamma \coloneqq \partial \left [u^{-1}(\mathbb{R} \setminus \mathscr{S}) \right ] \cap \Omega
$$
does not necessarily coincide with the set $\partial\{u>0\}$. Such an identification can only be guaranteed in special cases, namely when $\sigma (\cdot)$ is differentiable for $t>0$ and has a singularity at the origin. This scenario includes classical models such as the obstacle problem, and the quenching and cavity problems in a one-phase setting.
\end{remark}

\section{$C^1-$regularity}\label{C1 - sct}

In this section, we show that if the function $\sigma$ in the functional \eqref{main functional} is continuous, then local minimizers exhibit $C^1-$regularity at free boundary points. Moreover, we construct a $C^1-$modulus of continuity for $u$ depending only on the modulus of continuity of $\sigma$. Our approach employs a nonlinear renormalization algorithm for solutions, inspired by the methodology developed in \cite{APPT}, albeit applied in a substantially different context.

\begin{theorem}\label{C1 regularity at free boundary}
Let $u$ be a nonnegative local minimizer of \eqref{main functional} in $B_1$, with $\sigma$ being a modulus of continuity. Then, there exists a modulus of continuity $\omega \colon [0,1] \to [0,+\infty)$, depending on universal parameters and $\sigma$, such that, for any $x_0 \in \partial \{u>0\}\cap B_{1/2}$, 
$$
    |u(x)| \leq C\left(\|u\|_{L^\infty(B_{3/4})} + \sqrt{\|\sigma\|_{L^\infty(\mathbb{R})}} \, \right)|x-x_0|\,\omega(|x-x_0|), \quad x \in B_{1/4}(x_0),
$$
for some universal constant $C>0$.
\end{theorem}

\begin{proof}
Let $x_0 \in \partial \{u>0\} \cap B_{1/2}$ be an arbitrary point. Let $\epsilon = 2^{-2}$ in \Cref{flat lemma}, and consider the corresponding $\delta >0$. Up to replacing $u$ by
$$
    v(x) \coloneqq \frac{u(x_0 + 2^{-1}x)}{\|u\|_{L^\infty(B_{3/4})} + \sqrt{\delta^{-1}\|\sigma\|_{L^\infty(\mathbb{R})}}}, \quad x \in B_1,
$$
we may assume $u$ satisfies the assumption of \Cref{flat lemma} for this specific $\epsilon$.

We will prove that there is a sequence $(\mu_k)_{k \in \mathbb{N}} \subset (0,1)$ such that
$$
    \sup_{B_{2^{-k}}}u \leq 2^{-k}\left(\prod_{i=1}^{k}\mu_i\right), \quad \text{for every }\, k \in \mathbb{N}.
$$
This sequence will be built via an induction process. The first step, as we already are in a normalized setting, is to apply \Cref{flat lemma} to obtain
\begin{equation}\label{normalization step 1}
    \sup_{B_{2^{-1}}} u \leq  2^{-2} \leq 2^{-1} \mu_1,
\end{equation}
for some $2^{-1} \leq \mu_1 < 1$ to be chosen. To fix it, define the function
$$
    v_1(x) \coloneqq \frac{u\left(2^{-1}x\right)}{\mu_1 2^{-1}}, \quad x \in B_1,
$$
which is normalized due to \eqref{normalization step 1}. Moreover, $v_1$ is a local minimizer to
$$
    \mathcal{J}_1(w, B_1) \coloneqq \int_{B_1}\frac{1}{2}|Dw|^2 + \sigma_1(w)\, dx, \quad \text{with} 
    \quad \sigma_1(t) \coloneqq \mu_1^{-2}\sigma\left(\mu_1 2^{-1} t\right).
$$
At this point, we desire to apply \Cref{flat lemma} to $v_1$. As $v_1(0) = 0$, the only missing assumption is the smallness condition for $\sigma_1$. As $\sigma_1$ is a modulus of continuity, the smallness condition translates to ensuring $\sigma_1(1) \leq \delta$. This, in turn, is equivalent to
$$
    \mu_1^{-2}\sigma\left(\mu_1 2^{-1}\right) \leq \delta.
$$
To make this selection in a proper way, we build the following algorithm: if 
$$
   2^2\sigma\left(2^{-2}\right) \leq \delta,
$$
we pick $\mu_1 = 2^{-1}$. If this does not happen, then we pick $1 > \mu_1 > 2^{-1}$ so that
$$
    \mu_1^{-2}\sigma(\mu_1 2^{-1}) = \delta.
$$
Indeed, define
$$
    f(t) \coloneqq t^{-2} \sigma(2^{-1}t), \quad \text{for }\, t \in [2^{-1},1].
$$
This is a continuous function, such that 
$$f(1) = \sigma(2^{-1}) < \sigma(1) \leq \delta < f(2^{-1}),$$
and so, by the Intermediate Value Theorem, there should be $1 > \mu_1 > 2^{-1}$ such that 
$$
    f(\mu_1) = \delta.
$$
Now, we can apply \Cref{flat lemma} to $v_1$ to obtain
$$
    \sup_{B_{2^{-1}}} v_1 \leq 2^{-2} \ \Longrightarrow \ \sup_{B_{2^{-2}}} u \leq 2^{-2} \mu_1 \mu_2,
$$
for $\mu_1 \leq \mu_2 < 1$ to be chosen. Next, define
$$
    v_2(x) \coloneqq \frac{v_1\left(2^{-1}x\right)}{\mu_2 2^{-1}}, \quad x \in B_1.
$$
Again, it follows that $v_2$ is normalized, $v_2(0) = 0$ and it is a local minimizer of
$$
    \mathcal{J}_2(w,B_1) \coloneqq \int_{B_1} \frac12 |Dw|^2 + \sigma_2(w)dx, \quad \text{with} \quad \sigma_2(t) \coloneqq \mu_2^{-2}\sigma_1\left(\mu_2 2^{-1} t\right).
$$
We then follow the same routine: if
$$
    \mu_1^{-2}\sigma_1\left(\mu_1 2^{-1} \right) \leq \delta,
$$
then we pick $\mu_2 = \mu_1$. If not, then we pick $1 > \mu_2 > \mu_1$ so that
$$
   \mu_2^{-2}\sigma_1\left(\mu_2 2^{-1}\right) = \delta. 
$$
We can repeat this process inductively: we assume we have built up to $k$, that is, we have already found the parameter $\mu_k$ and the functions $\sigma_k$ and $v_k$ are already defined. For a parameter $1 > \mu_{k+1} \geq \mu_k$, we define the function
$$
    v_{k+1}(x) \coloneqq \frac{v_{k}(2^{-1}x)}{\mu_{k+1}2^{-1}}, \quad x \in B_1,
$$
which is normalized, $v_{k+1}(0) = 0$ and minimizes a functional with $\sigma_{k+1}(t) \coloneqq \mu_{k+1}^{-2}\sigma_k(\mu_{k+1}2^{-1}t)$. The choice of $\mu_{k+1}$ is as follows: if
$$
    \mu_{k}^{-2}\sigma_k(\mu_{k}2^{-1}) \leq \delta,
$$
then we pick $\mu_{k+1} = \mu_k$. If not, then we pick $\mu_{k+1}$ such that
$$
    \mu_{k+1}^{-2}\sigma_k\left(\mu_{k+1} 2^{-1}\right) = \delta.
$$
This concludes the construction of the sequence of parameters $\mu_k$, and so we have proven that
$$
    \sup_{B_{2^{-k}}}u \leq 2^{-k}\left(\prod_{i=1}^{k}\mu_i\right), \quad \text{for every }\, k \in \mathbb{N}.
$$
Define
$$
    a_k \coloneqq \prod_{i=1}^{k}\mu_i.
$$
By construction, $(a_k)_{k \in \mathbb{N}} \subset (0,1)$ and is decreasing. As a consequence, $a_k \to a_{\ast}$ for some $a_{\ast} \in [0,1)$. We prove that $a_{\ast} = 0$. Indeed, if after some point, the sequence stabilizes, meaning that
$$
    \mu_k = \mu_{k_0}, \quad \text{for all }\, k \geq k_0,
$$
then it readily follows that $a_{\ast} = 0$. Otherwise, we would have, for some subsequence $(k_j)_{j \in \mathbb{N}}$, that
$$
    a_{k_j}^{-2}\sigma(a_{k_j}2^{-k_j}) = \delta.
$$
Since $a_k \to a_{\ast}$, it also follows that $a_{k_j} \to a_{\ast}$, and so, passing to the limit in the equation above, gives
$$
    \sigma(a_{\ast} \cdot 0) = \delta a_{\ast}^2.
$$
But since $\sigma(0) = 0$, it follows that $a_{\ast} = 0$. 

From here, we can build the modulus of continuity as follows: pick $x \in B_{2^{-2}}$ and let $\rho = |x|$. There exists $k \in \mathbb{N}$ such that $2^{-(k+1)} < \rho \leq 2^{-k}$. Therefore, 
$$
    |u(x)| \leq \sup_{B_{\rho}} u \leq \sup_{B_{2^{-k}}}u \leq 2^{-k} a_k \leq 2\rho a_k,
$$
since $2^{-k} \leq 2\rho$. Taking into account that
$$
    k \leq \frac{\ln\left(\rho^{-1}\right)}{\ln(2)} < k+1,
$$
we may then define
$$
     \omega(\rho) \coloneqq 2 \prod_{i=1}^{\left\lfloor{\frac{\ln(\rho^{-1})}{\ln(2)}}\right\rfloor} \mu_i.
$$
Observe that $\omega$ is an increasing function and 
$$
    \lim_{\rho \to 0} \omega(\rho) = 2 \lim_{k \to \infty} a_k = 0.
$$
\end{proof}

\begin{remark} \label{rem}
We observe that if the $\sigma-$function is H\"older continuous, the modulus of continuity $\omega$ that we obtain in the previous theorem is
$$
    \omega(t) \approx t^{\frac{\gamma}{2-\gamma}},
$$
which is precisely the same modulus of continuity expected for the Alt--Phillips problem, see \Cref{holder sigma}. Indeed, say $\sigma(t) = \delta t^{\gamma}$ for simplicity, with the same $\delta$ of the previous proof. The first step of the algorithm is to select $\mu_1 \in [2^{-1},1)$ such that
$$
    \mu_1^{-2} \sigma(\mu_1 2^{-1}) \leq \delta,
$$
which is equivalent to
$$
    \mu_1^{-2} (\mu_1 2^{-1})^\gamma \leq 1.
$$
Isolating $\mu_1$ leads to
$$
    \mu_1^{-2+\gamma} \leq 2^\gamma.
$$
In particular,
$$
    \mu_1 = \left(2^{-1}\right)^{\frac{\gamma}{2-\gamma}}
$$
is a valid choice. This is the case where the sequence stabilizes and gives $\mu_k = \mu_1$, for every $k \in \mathbb{N}$. As a consequence, 
$$
    \prod_{i=1}^k \mu_i = \left(2^{-k}\right)^{\frac{\gamma}{2-\gamma}}, \quad \text{for every }\, k \in \mathbb{N},
$$
from which follows that $\omega(t) \approx t^{\frac{\gamma}{2-\gamma}}$. We point out that this goes beyond the $C^{1,\alpha}$ threshold. Indeed, if $\gamma > 1$, then
$$
    \frac{\gamma}{2-\gamma} > 1,
$$
which in particular implies the Hessian should also vanish at free boundary points.
\end{remark}

As pointed out in \Cref{rem}, the regularity regime of the minimizers at the free boundary varies according to the amount of smoothness $\sigma$ possesses. A corollary of the analysis carried out in \Cref{C1 regularity at free boundary} gives that whenever the modulus of continuity $\sigma$ is of order $\mathrm{o} (t^k )$, for $t>0$ and $k \in \mathbb{N}$, then solutions will be of class $C^{k+1}$, with a modulus of continuity, at the free boundary. Below, we consider the case $k=1$ to convey the idea.

\begin{proposition}
Let $u$ be a nonnegative local minimizer of \eqref{main functional} in $B_1$, with $\sigma$ being a modulus of continuity and satisfying
$\sigma(t) = \mathrm{o}(t)$, as $t \to 0$. Then, there exists a modulus of continuity $\omega \colon [0,1] \to [0,+\infty)$, depending on universal parameters and the modulus of continuity of $\sigma$, such that, for any $x_0 \in \partial \{u>0\}\cap B_{1/2}$, 
$$
    |u(x)| \leq C\left(\|u\|_{L^\infty(B_1)} + \sqrt{\|\sigma\|_{L^\infty(\mathbb{R})}}\, \right)|x-x_0|^2\,\omega(|x-x_0|), \quad x \in B_{1/4}(x_0),
$$
for some universal constant $C>0$.
\end{proposition}
 \begin{proof}
The proof follows along the lines of that of \Cref{C1 regularity at free boundary}, the only significant change being the choice of 
$$
    f(t) \coloneqq t^{-2} 2^2\sigma(2^{-2}t), \quad \text{for }\, t \in [2^{-1},1].
$$
To ensure $f(1) < \delta$, we use the fact that $\sigma = o(t)$. The rest of the proof goes through with the obvious modifications, by taking into account that we are in a $C^2-$regularity regime.
\end{proof}

\bigskip

{\small \noindent{\bf Acknowledgments.} This publication is based upon work supported by King Abdullah University of Science and Technology (KAUST) under Award No. ORFS-CRG12-2024-6430. JMU is partially supported by UID/00324 - Centre for Mathematics of the University of Coimbra.}

\bigskip

\bibliographystyle{amsplain, amsalpha}

\end{document}